\documentclass[a4paper]{amsart}

\usepackage[varg]{txfonts}
\usepackage{amsmath,amssymb,amsthm,enumerate}
\usepackage[all]{xy}
\usepackage[dvipdfmx]{graphicx}
\usepackage[dvips]{color}

\theoremstyle{plain}

\newtheorem{thm}{Theorem}[section]
\newtheorem{lem}[thm]{Lemma}
\newtheorem{prop}[thm]{Proposition}

\newtheorem{cor}[thm]{Corollary}

\theoremstyle{definition}

\newtheorem{defn}[thm]{Definition}
\newtheorem{rem}[thm]{Remark}

\newcommand{\R}{\mathbb{R}}

\newcommand{\calC}{\mathcal{C}}

\newcommand{\D}{\mathcal{D}}
\newcommand{\E}{\mathcal{E}}

\newcommand{\K}{\mathcal{K}}

\newcommand{\bfe}{\boldsymbol{e}}

\newcommand{\abs}[1]{\left\lvert {#1} \right\rvert}

\newcommand{\co}{\colon\thinspace}

\newcommand{\emb}[2]{\mathrm{Emb}(\R^{#2},\R^{#1})}

\newcommand{\id}{\mathrm{id}}

\newcommand{\Map}{\mathrm{Map}}

\numberwithin{equation}{section}
\numberwithin{figure}{section}

\allowdisplaybreaks[1]

\title{The space of short ropes and the classifying space of the space of long knots}
\author{Syunji Moriya and Keiichi Sakai}
\address{Department of Mathematics and Information Sciences, Osaka Prefecture University, Sakai, 599-8531, Japan}
\email{moriyasy@gmail.com}
\address{Faculty of Mathematics, Shinshu University, 3-1-1 Asahi, Matsumoto, Nagano 390-8621, Japan}
\email{ksakai@math.shinshu-u.ac.jp}
\date{\today}

\begin{document}
\maketitle

\begin{abstract}
We prove affirmatively the conjecture raised by J.~Mostovoy \cite{Mostovoy02}; the space of \emph{short ropes} is weakly homotopy equivalent to the classifying space of the topological monoid (or category) of long knots in $\R^3$.
We make use of techniques developed by S.~Galatius and O.~Randal-Williams \cite{GalatiusRandal-Williams10} to construct a manifold space model of the classifying space of the space of long knots, and we give an explicit map from the space of short ropes to the model in a geometric way.
\end{abstract}

\section{Introduction}\label{s:intro}
A \emph{long $j$--embedding in $\R^n$} is a smooth embedding $\R^j\hookrightarrow\R^n$ that coincides with the standard inclusion outside a compact set.
The space $\emb{n}{j}$ of all long $j$--embeddings in $\R^n$ equipped with the $C^{\infty}$--topology is widely studied in recent years,
in particular in the metastable range of dimensions.
Perhaps the space of \emph{long knots}, long $1$-embeddings in $\R^3$, is one of the most fascinating cases, but the dimension $(n,j)=(3,1)$ is not in the stable range
 and some methods for studying $\emb{n}{j}$ in high (co)dimensional cases yield only information on $K:=\pi_0(\emb{3}{1})$ when applied to $\emb{3}{1}$.
$K$ is just a free commutative monoid (and not a group) with respect to the connected-sum, and the group completion $\Omega B\emb{3}{1}$ would be strictly better from homotopy-theoretic view than $\emb{3}{1}$ itself.
In fact the group completion is a $2$-fold loop space,
since the little $2$--disks operad acts on $\emb{3}{1}$ (Budney \cite{Budney07}).
Moreover the group completion would be useful for study of (isotopy classes of) long knots since the natural map $\emb{3}{1}\to\Omega B\emb{3}{1}$ induces a monomorphism on $\pi_0$.

From this viewpoint the result of Mostovoy \cite{Mostovoy02} is very curious though it is also concerned with $K$.
A \emph{parametrized short rope} is a smooth embedding $\rho\co[0,1]\hookrightarrow\R^3$ of length $<3$ such that $\rho(i)=(i,0,0)$ for $i=0,1$.
Mostovoy has proved that the fundamental group of the space $B_2$ of parametrized short ropes is isomorphic to $\pi_1BK$, the group completion of $K$.
This leads us to the question \cite[Conjecture~1]{Mostovoy02}: is the space $B_2$ the classifying space $B\emb{3}{1}$ of $\emb{3}{1}$ ?
Our main result asserts that this is the case.

\begin{thm}[Corollary~\ref{cor:BK=R}, Theorem~\ref{thm:ropes_equiv}]\label{thm:intro_main_1}
Mostovoy's space of parametrized short ropes is weakly homotopy equivalent to the classifying space of the space of long knots.
\end{thm}

One of the main ingredients in the proof of Theorem~\ref{thm:intro_main_1} is the technique of Galatius and Randal-Williams \cite{GalatiusRandal-Williams10}.
It enables us to construct a model of $B\emb{3}{1}$．
The model is a space of certain $1$--dimensional submanifolds in $\R^3$ whose connected components are non-compact closed subspaces of $\R^3$ (see Definition~\ref{def:manifold_space}).
We prove Theorem~\ref{thm:intro_main_1} by introducing the notion of \emph{reducible ropes} (see Definition~\ref{def:short_rope}) and by comparing the manifold space model with the space of short ropes through reducible ropes:

\begin{thm}[Corollary~\ref{cor:BK=R}, Theorem~\ref{thm:ropes_equiv}]\label{thm:intro_main_2}
The manifold space model and Mostovoy's space of parametrized short ropes are both weakly homotopy equivalent to the space of reducible ropes.
\end{thm}

It is very interesting that we can realize the weak equivalence from the manifold space model to the space of reducible ropes as a ``cut-off map'' which is explicit and geometric.
Therefore Mostovoy's space of short ropes and the space of reducible ropes would serve as tools to study $B\emb{3}{1}$ in a geometric way.

\section{Manifold space model of the classifying space of the space of long knots}\label{s:long_knot_space}
\subsection{Notations}
Throughout this paper $D^m$ and $\overline{D}{}^m$ stand respectively for the open and closed unit $m$-disks;
\[
 D^m:=\{p\in\R^m\mid\abs{p}<1\},\quad
 \overline{D}{}^m:=\{p\in\R^m\mid\abs{p}\le 1\}.
\]

For a $1$--dimensional manifold $M\subset\R^1\times D^2$ and a subset $A\subset\R^1$,
let
\[
 M|_A:=M\cap(A\times D^2).
\]
For a one point set $A=\{T\}$,
we simply write $M|_T$ for $M|_{\{T\}}$.

\begin{defn}\label{def:transverse_cylindrical}
A $1$--dimensional manifold $M\subset\R^1\times D^2$ is said to be
\begin{itemize}
\item
	\emph{reducible} \emph{at} $T\in\R^1$ if $M$ intersects $\{T\}\times D^2$ transversely in one point set.
\item
	\emph{strongly reducible} \emph{at} $T\in\R^1$ if $M|_T$ is one point set and there exists an $\epsilon>0$ satisfying
	\[
	 M|_{(T-\epsilon, T+\epsilon)}=(T-\epsilon,T+\epsilon)\times\{p_{23}(M|_T)\},
	\]
	where $p_{23}\co\R^1\times D^2\to D^2$ is the projection.
\end{itemize}
\end{defn}

\begin{rem}
The word ``reducible'' indicates that the manifold looks like a ``connected sum'' of two $1$--manifolds.
But the meaning is different from that in knot theory, in that a reducible manifold does not need to split into a connected sum of nontrivial knots.
\end{rem}

\subsection{The category $\K$ of long knots}
First we define the space $\psi$ that we have referred to in Section \ref{s:intro} as the \emph{manifold space model}.
\begin{defn}\label{def:manifold_space}
Let $\psi$ be the set of $1$--dimensional submanifolds $M\subset\R^1\times D^2$ such that
\begin{itemize}
\item
	$\partial M=\emptyset$,	
\item
	each connected component of $M$ is a closed, non-compact subspace in $\R^3$, and
\item
	there exists at least one $T\in\R$ such that $M$ is reducible at $T$
\end{itemize}
(see Figure~\ref{fig:psi}).
\begin{figure}
\begin{center}
\input{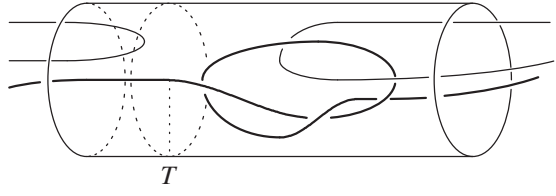}
\end{center}\caption{An element of $\psi$; the long component is drawn with a thick curve}
\label{fig:psi}
\end{figure}
The above conditions imply that $M\in\psi$ contains exactly one connected component $M_0$ satisfying $M_0|_t\ne\emptyset$ for any $t\in\R^1$.
Such a component is said to be \emph{long}.
It can also be seen that the other connected components (if they exist) are \emph{long on exactly one side};
we say a component $M_1$ is \emph{long on the left} (resp.~\emph{right})
if there exists $T\in\R^1$ such that $M_1|_s\ne\emptyset$ for any $s\le T$ (resp.~$s\ge T$)
but $M_1|_{(T,\infty)}=\emptyset$ (resp.~$M_1|_{(-\infty,T)}=\emptyset$).
The set $\psi$ is topologized as a subspace of $\psi(3,1)$ from Galatius and Randal-Williams \cite[Section 3.1]{GalatiusRandal-Williams10} (without any ``tangential data'').
\end{defn}

\begin{rem}\label{rem:K-topology}
Roughly speaking, two manifolds $M,N\in\psi$ are ``close to each other if they are close in a compact subspace of $\R^3$''.
A bit more precisely, for $M\in\psi$, the set of manifolds whose intersections with some compact subspace of $\R^3$ are obtained by shifting $M$ along small normal sections to $M$, is a basic open neighborhood of $M$ in $\psi$.
It is worth mentioning the following example:
Let $\alpha\co[0,1)\to\R_{\ge 0}$ be a monotonically increasing function with $\alpha(0)=0$ and $\lim_{t\to 1}\alpha(t)=\infty$, and $M(t)\in\psi$ ($0\le t<1$) a continuous family satisfying $M(t)|_{[-\alpha(t),\,\alpha(t)]}=[-\alpha(t),\,\alpha(t)]\times \{(0,0)\}$.
Then $M(t)$ converges to the trivial long knot $\R^1\times\{(0,0)\}$ in this topology as $t$ tends to $1$ (see also \cite[Example~2.2]{GalatiusRandal-Williams10}).
\end{rem}

\begin{rem}\label{rem:separated}
For any $M\in\psi$ there exists $T\in\R^1$ such that all the components of $M$ that are long on the left (resp.\ right) are contained in $(-\infty,T)\times D^2$ (resp.~ $(T,\infty)\times D^2$).
\end{rem}

\begin{defn}\label{def:cat_long_knot}
We define the \emph{category} $\K$ \emph{of long knots} as follows.
The space of objects of $\K$ is $D^2$ with the usual topology.
A non-identity morphism from $p$ to $q$ is a pair $(T,M)$,
where $T>0$ and $M\in\psi$ is a \emph{long knot from $p$ to} $q$, namely a connected $1$--manifold (and hence long) that is strongly reducible at any $t\in(-\infty,\epsilon)\cup(T-\epsilon,\infty)$ for some $\epsilon>0$;
\[
 M|_{(-\infty,\epsilon)}=(-\infty,\epsilon)\times\{ p\},\quad
 M|_{(T-\epsilon,\infty)}=(T-\epsilon,\infty)\times\{ q\}.
\]
The identity morphism $\id\co p\to p$ is given by $(0,\R^1\times\{p\})$.
The total space $\bigcup_{p,q}\Map_{\K}(p,q)$ of all morphisms is topologized as a subspace of $(\{0\}\sqcup\R^1_{>0})\times\psi$,
where $\{0\}\sqcup\R^1_{>0}$ is a disjoint union.
The composition $\circ\co\Map_{\K}(q,r)\times\Map_{\K}(p,q)\to\Map_{\K}(p,r)$ is defined by
\[
 (T_1,M_1)\circ(T_0,M_0):=(T_0+T_1,M_0|_{(-\infty,T_0]}\cup(M_1|_{[0,\infty)}+T_0\bfe_1)),
\]
where $\bfe_1=(1,0,0)\in\R^3$ and $+T\bfe_1$ stands for the translation by $T$ in the direction of $\R^1$.
\end{defn}

In this section we show that $B\K$ (see Section \ref{ss:BC} for the definition) is weakly equivalent to $\psi$.
The following posets play roles as interfaces between them.

\begin{defn}
Define a poset $\D$ by
\[
 \D:=\{(T,M)\in\R^1\times\psi\mid
 M\text{ is reducible at }T\}
\]
and topologize $\D$ as a subspace of $\R^1\times\psi$.
Define the partial order $\le$ on $\D$ so that $(T,M)<(T',M')$ if and only if $M=M'$ and $T<T'$.
We regard $\D$ as a small category in the usual way,
namely $\Map_{\D}(x,y)$ is a one point set $\{(x,y)\}$ if $x\le y$, and $\emptyset$ otherwise.
The total space of all morphisms is topologized as a subspace of $\bigl(\Delta\sqcup(\R^1\times\R^1\setminus\Delta)\bigr)\times\psi$,
where $\Delta:=\{(x,x)\in\R^1\times\R^1\}$ is the diagonal set.

Define $\D^{\perp}$ as a subposet of $\D$ consisting of $(T,M)$ with $M$ strongly reducible at $T$.
\end{defn}

\subsection{Classifying spaces of categories}\label{ss:BC}
Here we recall the general definition of classifying spaces of topological categories.

For a topological category $\calC$, its \emph{nerve} is the simplicial space whose level $l$ space $N_l\calC$ consists of sequences of composable $l$ morphisms $(x_0\xrightarrow{f_1}x_1\xrightarrow{f_2}\dotsb\xrightarrow{f_l}x_l)$ in $\calC$ and is topologized as a subspace of the $l$-th power of the total space of all morphisms in $\calC$.
By definition $N_0\calC$ is the space of objects in $\calC$.
The face maps are given by the compositions, and the degeneracy maps are given by inserting the identity morphisms.
The \emph{classifying space} $B\calC$ of $\calC$ is defined as the geometric realization of $N_*\calC$;
\[
 B\calC:=\abs{N_*\calC}:=\Bigl(\bigsqcup_{l\ge 0}(N_l\calC\times\Delta^l)\Bigr)/\mathord{\sim},
\]
where $\Delta^l:=\{(\lambda_0,\dotsc,\lambda_l)\in[0,1]^{l+1}\mid\sum_i\lambda_i=1\}$ is the standard $l$-simplex.
The relation $\sim$ is defined so that, for any order preserving map $\sigma\co\{0,\dotsc,l\pm 1\}\to\{0,\dotsc,l\}$,
\begin{equation}\label{eq:relation}
 N_{l\pm 1}\calC\times\Delta^{l\pm 1}\ni(\sigma^*f,\lambda) \sim (f,\sigma_*\lambda)\in N_l\calC\times\Delta^l
\end{equation}
where $\sigma_*$ and $\sigma^*$ are the induced maps on (co)simplicial spaces.

Recall from Segal \cite{Segal74} a sufficient condition for a simplicial map to induce a homotopy equivalence on geometric realizations.

\begin{defn}[{\cite[Definition~A.4]{Segal74}}]
We say a simplicial space $A_*$ is \emph{good} if $s_iA_l\hookrightarrow A_{l+1}$ is a closed cofibration for each $l$ and $0\le i\le l$, where $s_i$ stands for the $i$-th degeneracy map.
\end{defn}

\begin{lem}[{\cite[Proposition~A.1]{Segal74}}]\label{lem:levelwise_equivalence}
Let $A_*$ and $B_*$ be good simplicial spaces.
Suppose there exists a simplicial map $f_*\co A_*\to B_*$ which is a levelwise homotopy equivalence, that is $f_l\co A_l\to B_l$ is a homotopy equivalence for each $l$.
Then $f$ induces a homotopy equivalence $\abs{f_*}\co\abs{A_*}\xrightarrow{\simeq}\abs{B_*}$ on the geometric realizations.
\end{lem}

\subsection{The classifying space of $\K$}
Notice that any element of $N_l\D$ (resp.\ $N_l\D^{\perp}$), $l\ge 0$, can be expressed as a pair $(T_0\le\dotsb\le T_l;M)$,
where $M\in\psi$ is reducible (resp.\ strongly reducible) at each $T_i$.
Similarly any element of $N_l\K$ ($l\ge 1$) is of the form $(0\le T_1\le\dotsb\le T_l;M)$, where $M$ is a long knot that is strongly reducible at each $T_i$.

\begin{lem}\label{lem:good}
The simplicial spaces $N_*\K$, $N_*\D$ and $N_*\D^{\perp}$ are good.
\end{lem}
\begin{proof}
For $0\le i\le l$,
$s_iN_l\K=\{(0\le T_1\le \dotsb\le T_{l+1};M)\mid T_i=T_{i+1}\}\subset N_{l+1}\K$ (here $T_0:=0$)
is a union of connected components of sequences involving identity morphisms,
and hence the pair $(N_{l+1}\K,s_iN_l\K)$ has the homotopy extension property.
The proofs for $N_*\D$ and $N_*\D^{\perp}$ are the same.
\end{proof}

\begin{prop}\label{prop:zig-zag}
There exists a zig-zag of levelwise homotopy equivalences
$N_*\K\leftarrow N_*\D^{\perp}\to N_*\D$.
Consequently $B\K\leftarrow B\D^{\perp}\to B\D$ are all homotopy equivalences.
\end{prop}
\begin{proof}
The proof is the same as in Galatius and Randal-Williams \cite[Theorem~3.9]{GalatiusRandal-Williams10}.
That $B\D^{\perp}\to B\D$ induced by the inclusion is a homotopy equivalence follows from \cite[Lemma~3.4]{GalatiusRandal-Williams10},
which states that, for any $(T_0\le\dotsb\le T_l;M)\in N_l\D$, $M$ can be modified to be strongly reducible at $T_i$ in a canonical way.

Define the functor $F\co\D^{\perp}\to\K$ on objects by $(T,M)\mapsto M|_T$,
and on morphisms by
\[
 F(T_0\le\dotsb\le T_l;M):=(0\le T_1-T_0\le\dotsb\le T_l-T_0;\overline{M|_{[T_0,T_l]}}-T_0\bfe_1),
\]
where $\overline{M|_{[T_0,T_l]}}$ is the \emph{long-extension} of $M|_{[T_0,T_l]}$ (see Figure~\ref{fig:long_extension}),
namely
\begin{equation}\label{eq:long-extension}
 \overline{M|_{[T_0,T_l]}}:=
 \bigl((-\infty,T_0]\times\{p_{23}(M|_{T_0})\}\bigr)\cup
 M|_{[T_0,T_l]}\cup
 \bigl([T_l,\infty)\times\{p_{23}(M|_{T_l})\}\bigr),
\end{equation}
where $p_{23}\co\R\times D^2\to D^2$ is the second projection
(\eqref{eq:long-extension} is the same as $(\varphi_{\infty}(T_0,T_l)\times\id)^{-1}(M)$ in \cite[Section 3.2]{GalatiusRandal-Williams10}).
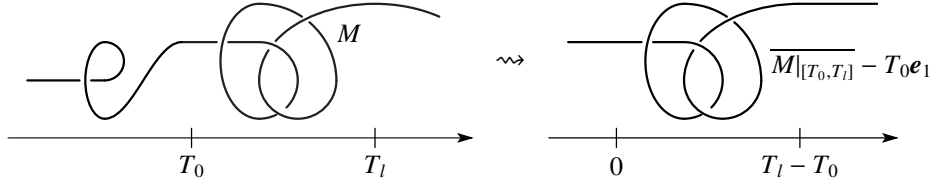
\begin{figure}
\begin{center}
{\unitlength 0.1in%
\begin{picture}(46.0000,7.8500)(7.0000,-11.8500)%
%
\special{pn 13}%
\special{ar 2000 800 200 200 4.7123890 0.7853982}%
%
\special{pn 13}%
\special{ar 2200 800 200 200 6.2831853 3.1415927}%
%
\special{pn 13}%
\special{ar 2000 800 400 400 5.4210153 6.2831853}%
%
\special{pn 13}%
\special{ar 2000 800 400 400 4.7123890 5.3003916}%
%
\special{pn 13}%
\special{ar 2600 800 600 400 3.6820122 5.3003916}%
%
\special{pn 13}%
\special{ar 2600 800 600 400 3.1415927 3.5464844}%
%
\special{pn 13}%
\special{pa 1600 600}%
\special{pa 1780 600}%
\special{fp}%
%
\special{pn 13}%
\special{pa 1840 600}%
\special{pa 2000 600}%
\special{fp}%
%
\special{pn 13}%
\special{ar 2000 700 200 300 1.5707963 4.7123890}%
%
\special{pn 13}%
\special{ar 2000 800 200 200 1.1071487 1.5707963}%
%
\special{pn 8}%
\special{pa 700 1100}%
\special{pa 3100 1100}%
\special{fp}%
\special{sh 1}%
\special{pa 3100 1100}%
\special{pa 3033 1080}%
\special{pa 3047 1100}%
\special{pa 3033 1120}%
\special{pa 3100 1100}%
\special{fp}%
%
\special{pn 8}%
\special{pa 1650 1150}%
\special{pa 1650 1050}%
\special{fp}%
%
\special{pn 8}%
\special{pa 2600 1150}%
\special{pa 2600 1050}%
\special{fp}%
\put(16.5000,-12.5000){\makebox(0,0){$T_0$}}%
\put(26.0000,-12.5000){\makebox(0,0){$T_l$}}%
%
\special{pn 13}%
\special{ar 4200 800 200 200 4.7123890 0.7853982}%
%
\special{pn 13}%
\special{ar 4400 800 200 200 6.2831853 3.1415927}%
%
\special{pn 13}%
\special{ar 4200 800 400 400 5.4210153 6.2831853}%
%
\special{pn 13}%
\special{ar 4200 800 400 400 4.7123890 5.3003916}%
%
\special{pn 13}%
\special{ar 4800 800 600 400 3.6820122 4.7123890}%
%
\special{pn 13}%
\special{ar 4800 800 600 400 3.1415927 3.5464844}%
%
\special{pn 13}%
\special{pa 3600 600}%
\special{pa 3980 600}%
\special{fp}%
%
\special{pn 13}%
\special{pa 4040 600}%
\special{pa 4200 600}%
\special{fp}%
%
\special{pn 13}%
\special{ar 4200 700 200 300 1.5707963 4.7123890}%
%
\special{pn 13}%
\special{ar 4200 800 200 200 1.1071487 1.5707963}%
%
\special{pn 13}%
\special{pa 4800 400}%
\special{pa 5200 400}%
\special{fp}%
\put(24.0000,-6.0000){\makebox(0,0)[lb]{$M$}}%
\put(46.5000,-8.0000){\makebox(0,0)[lb]{$\overline{M|_{[T_0,T_l]}}-T_0\bfe_1$}}%
%
\special{pn 8}%
\special{pa 3500 1100}%
\special{pa 5300 1100}%
\special{fp}%
\special{sh 1}%
\special{pa 5300 1100}%
\special{pa 5233 1080}%
\special{pa 5247 1100}%
\special{pa 5233 1120}%
\special{pa 5300 1100}%
\special{fp}%
%
\special{pn 8}%
\special{pa 3850 1150}%
\special{pa 3850 1050}%
\special{fp}%
%
\special{pn 8}%
\special{pa 4800 1150}%
\special{pa 4800 1050}%
\special{fp}%
\put(38.5000,-12.5000){\makebox(0,0){$0$}}%
\put(48.0000,-12.5000){\makebox(0,0){$T_l-T_0$}}%
\put(33.0000,-7.0000){\makebox(0,0){$\rightsquigarrow$}}%
\special{pn 13}%
\special{pa 1200 1000}%
\special{pa 1205 1000}%
\special{pa 1210 999}%
\special{pa 1215 999}%
\special{pa 1220 998}%
\special{pa 1245 988}%
\special{pa 1255 982}%
\special{pa 1265 974}%
\special{pa 1270 971}%
\special{pa 1275 966}%
\special{pa 1280 962}%
\special{pa 1295 947}%
\special{pa 1300 941}%
\special{pa 1305 936}%
\special{pa 1320 918}%
\special{pa 1330 904}%
\special{pa 1335 898}%
\special{pa 1350 877}%
\special{pa 1355 869}%
\special{pa 1360 862}%
\special{pa 1365 854}%
\special{pa 1370 847}%
\special{pa 1380 831}%
\special{pa 1385 824}%
\special{pa 1415 776}%
\special{pa 1420 769}%
\special{pa 1430 753}%
\special{pa 1435 746}%
\special{pa 1440 738}%
\special{pa 1445 731}%
\special{pa 1450 723}%
\special{pa 1465 702}%
\special{pa 1470 696}%
\special{pa 1480 682}%
\special{pa 1495 664}%
\special{pa 1500 659}%
\special{pa 1505 653}%
\special{pa 1520 638}%
\special{pa 1525 634}%
\special{pa 1530 629}%
\special{pa 1535 626}%
\special{pa 1545 618}%
\special{pa 1555 612}%
\special{pa 1580 602}%
\special{pa 1585 601}%
\special{pa 1590 601}%
\special{pa 1595 600}%
\special{pa 1600 600}%
\special{fp}%
%
\special{pn 13}%
\special{ar 1200 800 100 200 1.5707963 4.7123890}%
%
\special{pn 13}%
\special{ar 1200 700 100 100 4.7123890 1.5707963}%
%
\special{pn 13}%
\special{pa 1200 800}%
\special{pa 1130 800}%
\special{fp}%
%
\special{pn 13}%
\special{pa 1070 800}%
\special{pa 800 800}%
\special{fp}%
\end{picture}}%
\end{center}
\caption{The functor $F$ from the proof of Proposition~\ref{prop:zig-zag}; cut-off and long-extension}
\label{fig:long_extension}
\end{figure}
Notice that $M|_{[T_0,T_l]}$ is a connected subspace of the long component of $M$ (see Remark~\ref{rem:separated}),
and its long extension is also connected.
This induces a map $F\co N_*\D^{\perp}\to N_*\K$ of simplicial spaces.

We have a map $G\co N_*\K\to N_*\D^{\perp}$, defined in level $0$ by
$G(p):=(0,\R^1\times\{p\})$,
and by the natural inclusion in positive levels (letting $T_0:=0$).
This is just a simplicial map up to homotopy (in levels $0$ and $1$),
but is a levelwise homotopy inverse to $F$;
the composite $F\circ G$ is the identity,
and the other composite $G\circ F$ is given by
\[
 G\circ F(T_0\le\dotsb\le T_l;M)=(0\le T_1-T_0\le\dotsb\le T_l-T_0;\overline{M_0|_{[T_0,T_l]}})
\]
which is isotopic to the identity via the same homotopy as the one exhibited in the last line in the proof of \cite[Theorem~3.9]{GalatiusRandal-Williams10} (see Figure~\ref{fig:h_s}).
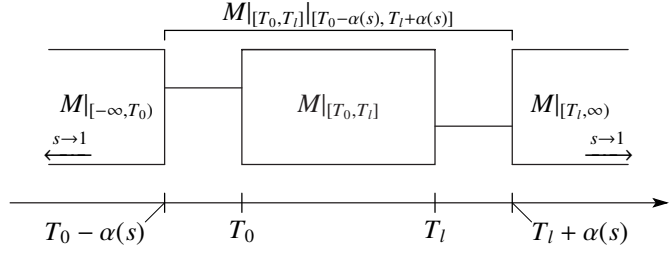
\begin{figure}
\begin{center}
{\unitlength 0.1in%
\begin{picture}(36.4000,11.6500)(7.6000,-15.0000)%
%
\special{pn 8}%
\special{pa 2200 600}%
\special{pa 3200 600}%
\special{pa 3200 1200}%
\special{pa 2200 1200}%
\special{pa 2200 600}%
\special{pa 3200 600}%
\special{fp}%
\put(27.0000,-9.0000){\makebox(0,0){$M|_{[T_0,T_l]}$}}%
%
\special{pn 8}%
\special{pa 2200 800}%
\special{pa 1800 800}%
\special{fp}%
%
\special{pn 8}%
\special{pa 3200 1000}%
\special{pa 3600 1000}%
\special{fp}%
%
\special{pn 8}%
\special{pa 1200 600}%
\special{pa 1800 600}%
\special{fp}%
%
\special{pn 8}%
\special{pa 1800 600}%
\special{pa 1800 1200}%
\special{fp}%
%
\special{pn 8}%
\special{pa 1800 1200}%
\special{pa 1200 1200}%
\special{fp}%
%
\special{pn 8}%
\special{pa 4200 600}%
\special{pa 3600 600}%
\special{fp}%
%
\special{pn 8}%
\special{pa 3600 600}%
\special{pa 3600 1200}%
\special{fp}%
%
\special{pn 8}%
\special{pa 3600 1200}%
\special{pa 4200 1200}%
\special{fp}%
%
\special{pn 8}%
\special{pa 1000 1400}%
\special{pa 4400 1400}%
\special{fp}%
\special{sh 1}%
\special{pa 4400 1400}%
\special{pa 4333 1380}%
\special{pa 4347 1400}%
\special{pa 4333 1420}%
\special{pa 4400 1400}%
\special{fp}%
%
\special{pn 8}%
\special{pa 2200 1450}%
\special{pa 2200 1350}%
\special{fp}%
%
\special{pn 8}%
\special{pa 3200 1450}%
\special{pa 3200 1350}%
\special{fp}%
\put(22.0000,-15.5000){\makebox(0,0){$T_0$}}%
\put(32.0000,-15.5000){\makebox(0,0){$T_l$}}%
%
\special{pn 8}%
\special{pa 3600 1450}%
\special{pa 3600 1350}%
\special{fp}%
\put(39.0000,-9.0000){\makebox(0,0){$M|_{[T_l,\infty)}$}}%
\put(15.0000,-9.0000){\makebox(0,0){$M|_{[-\infty,T_0)}$}}%
%
\special{pn 8}%
\special{pa 1800 1450}%
\special{pa 1800 1350}%
\special{fp}%
\put(27.0000,-4.0000){\makebox(0,0){$\overline{M|_{[T_0,T_l]}}|_{[T_0-\alpha(s),\,T_l+\alpha(s)]}$}}%
%
\special{pn 8}%
\special{pa 1800 500}%
\special{pa 1800 550}%
\special{fp}%
%
\special{pn 8}%
\special{pa 3600 500}%
\special{pa 1800 500}%
\special{fp}%
%
\special{pn 8}%
\special{pa 3600 500}%
\special{pa 3600 550}%
\special{fp}%
\put(41.0000,-11.0000){\makebox(0,0){$\xrightarrow{s\to 1}$}}%
\put(13.0000,-11.0000){\makebox(0,0){$\xleftarrow{s\to 1}$}}%
\put(17.0000,-15.0000){\makebox(0,0)[rt]{$T_0-\alpha(s)$}}%
%
\special{pn 4}%
\special{pa 1700 1500}%
\special{pa 1800 1400}%
\special{fp}%
\put(37.0000,-15.0000){\makebox(0,0)[lt]{$T_l+\alpha(s)$}}%
%
\special{pn 4}%
\special{pa 3700 1500}%
\special{pa 3600 1400}%
\special{fp}%
\end{picture}}%
\end{center}
\caption{The homotopy in the proof of Proposition~\ref{prop:zig-zag} from $G\circ F$ to the identity; where $\alpha(s)\to\infty$ ($s\nearrow 1$)}
\label{fig:h_s}
\end{figure}
This homotopy firstly extends $M|_{(T_0-\epsilon,T_0]}$ and $M|_{[T_l,T_l+\epsilon)}$ respectively to left and right so that $M|_{(-\infty,T_0)}$ and $M|_{(T_l,\infty)}$ (in which all the one-side long components are contained) escape respectively to ``$\{\mp\infty\}\times D^2$''.
Then they ``vanish'' at $s=1$ by definition of the topology of $\psi$, see Remark~\ref{rem:K-topology}).
Simultaneously this homotopy translates the manifold by $-T_0$ in the direction of $\R^1$.
This homotopy keeps manifolds strongly reducible at each $T_i$.

Therefore $F\co N_*\D^{\perp}\to N_*\K$ is a levelwise homotopy equivalence of good simplicial spaces (Lemma~\ref{lem:good}), and $B\D^{\perp}\to B\K$ is a homotopy equivalence by Lemma~\ref{lem:levelwise_equivalence}.
\end{proof}

Following Galatius and Randal-Williams \cite{GalatiusRandal-Williams10}, we denote the element of $B\D$ represented by $\bigl((T_0\le\dotsb\le T_l;M),(\lambda_0,\dotsc,\lambda_l)\bigr)\in N_l\D\times\Delta^l$ as a formal sum $\sum_{0\le i\le l}\lambda_iT_i$ (this notation is compatible with the relation \eqref{eq:relation}).

\begin{thm}\label{thm:BK=psi_s}
The forgetful map $u\co B\D\to\psi$ given by $\sum_i\lambda_iT_i\mapsto M$ is a weak homotopy equivalence.
Thus $B\K$ is weakly equivalent to $\psi$.
\end{thm}
\begin{proof}
The proof is the same as that of \cite[Theorem~3.10]{GalatiusRandal-Williams10}:
Given the following commutative diagram of strict arrows,
\[
 \xymatrix{
  \partial \overline{D}{}^m\ar[r]^-{\hat{f}}\ar@{^(->}[d] & B\D\ar[d]^-u \\
  \overline{D}{}^m\ar[r]^-{f}\ar@{.>}[ru]^-{g} & \psi
 }
\]
we find a dotted $g\co\overline{D}{}^m\to B\D$ that makes the diagram commutative.
This means that the relative homotopy group $\pi_m(\psi',B\D)$ ($\psi'$ is the mapping cylinder of $u$) vanishes for all $m$, and $u$ induces an isomorphism of homotopy groups in any dimension.

For $a\in\R$ let $U_a:=\{x\in \overline{D}{}^m\mid f(x)\in\psi\text{ is reducible at }a\}$.
This is an open subspace of $\overline{D}{}^m$ and $\{U_a\}_{a\in\R}$ is an open covering of $\overline{D}{}^m$
because, by definition, such an $a$ exists for any $M\in\psi$.
So by compactness we can pick finitely many $a_0<\dotsb<a_k$ such that $\{U_{a_i}\}_{0\le i\le k}$ covers $\overline{D}{}^m$.
Pick a partition of unity $\{\lambda_i\co\overline{D}{}^m\to[0,1]\}_{0\le i\le k}$ subordinate to the cover.
Using $\lambda_i$ as a formal coefficient of $a_i$ gives a map
\[
 \hat{g}\co\overline{D}{}^m\to B\D,\quad
 \hat{g}(x):=\sum_{0\le i\le k}\lambda_i(x)a_i
\]
(represented by elements in $N_k\D\times\Delta^k$)
which lifts $f$, namely $u\circ\hat{g}=f$.
Now we produce a homotopy $h\co[0,1]\times\partial\overline{D}{}^m\to B\D$ such that $h(0,-)=\hat{g}|_{\partial \overline{D}{}^m}(-)$, $h(1,-)=\hat{f}(-)$ and $h(s,-)$ lifts $f|_{\partial\overline{D}{}^m}$ for all $s$;
if such an $h$ exists, then we can define the desired map $g$ by
\[
 g(x):=
 \begin{cases}
 \hat{g}(2x) & \abs{x}\le 1/2,\\
 h(2\abs{x}-1,x/\abs{x}) & \abs{x}\ge 1/2.
 \end{cases}
\]
Since $\hat{f}$ is also a lift of $f|_{\partial \overline{D}{}^m}$, we may suppose that $\hat{f}$ is of the form
\[
 \hat{f}(x)=\sum_{0\le i\le l}\mu_i(x)b_i
\]
for some $\mu_0,\dotsc,\mu_l\ge 0$, $\sum_i\mu_i(x)=1$ and $b_0<\dotsb<b_l$
(underlying manifolds $f(x)$ and $u(\hat{f}(x))$ are the same).
Let $c_0<\dotsb<c_n$ be the re-ordering of the set $\{a_i\}_i\cup\{ b_j\}_j$ in ascending order.
Using the relation \eqref{eq:relation} we can write $\hat{g}|_{\partial\overline{D}{}^m}$ and $\hat{f}$ as
\begin{align*}
 \hat{g}|_{\partial\overline{D}{}^m}(x)&=\sum_{0\le i\le n}\alpha_i(x)c_i\quad\text{for some}\quad \alpha_0,\dotsc,\alpha_n\ge 0,\quad\sum_i\alpha_i=1,\\
 \hat{f}(x)&=\sum_{0\le i\le n}\beta_i(x)c_i\quad\text{for some}\quad \beta_0,\dotsc,\beta_n\ge 0,\quad\sum_i\beta_i=1
\end{align*}
(represented by elements in $N_n\D\times\Delta^n$).
We define $h$ using the affine structure on the fibers of $u$;
\[
 h(s,x):=s\hat{g}|_{\partial\overline{D}{}^m}(x)+(1-s)\hat{f}(x):=\sum_{0\le i\le n}(s\alpha_i(x)+(1-s)\beta_i(x))c_i.\qedhere
\]
\end{proof}

\begin{rem}
We have topologized the spaces of morphisms of various categories so that the identity morphisms form disjoint components,
as was also done in \cite{GalatiusRandal-Williams10}.
We may instead topologize the total space of morphisms in $\K$ (resp.~$\D$) as a subspace of $[0,\infty)\times\psi$ (resp.~$\R\times\R\times\psi$) and with the latter topology we can prove the similar results to the above.
An advantage of the former topology is that it makes the proof of goodness of the nerves easier.
\end{rem}
\section{The space of reduced ropes}\label{s:rope}
In this section we show that the conjecture of Mostovoy is true.
We first characterize the weak homotopy type of $\psi$ as that of the space of \emph{reducible ropes}, and then prove that the space of reducible ropes is weakly equivalent to the space of Mostovoy's short ropes.

\subsection{$B\K$ and the space of reducible ropes}
\begin{defn}[Mostovoy \cite{Mostovoy02}]\label{def:short_rope}
A \emph{rope} is a compact, connected $1$--dimensional submanifold $r\subset\R^1\times D^2$ with non-empty boundary $\partial r=\{\partial_0r,\partial_1r\}$, $\partial_ir\in\{i\}\times D^2$.
Let $R$ be the set of all ropes that are reducible at some $t\in(0,1)$ (see Figure~\ref{fig:reducible_ropes}),
topologized as a subspace of $\mathrm{Emb}([0,1],\R\times D^2)/\mathrm{Diff}^+([0,1])$.
\end{defn}
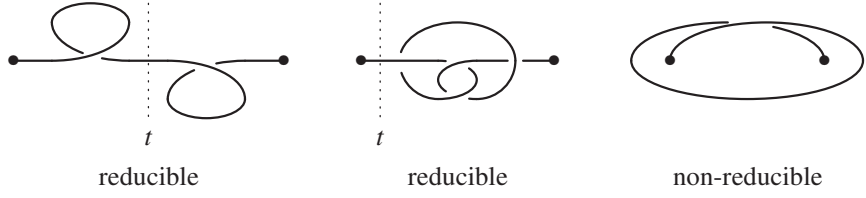
\begin{figure}
\begin{center}
{\unitlength 0.1in%
\begin{picture}(46.3700,8.2700)(7.6300,-15.2700)%
%
\special{pn 13}%
\special{ar 1200 800 400 200 6.2831853 1.5707963}%
%
\special{pn 13}%
\special{ar 1600 800 400 200 2.2142974 3.1415927}%
%
\special{pn 13}%
\special{ar 1400 800 200 100 3.1415927 6.2831853}%
%
\special{pn 13}%
\special{ar 1600 800 400 200 1.5707963 1.9295670}%
%
\special{pn 13}%
\special{pa 1000 1000}%
\special{pa 1200 1000}%
\special{fp}%
%
\special{pn 13}%
\special{pa 1600 1000}%
\special{pa 1800 1000}%
\special{fp}%
%
\special{pn 13}%
\special{ar 1800 1200 400 200 4.7123890 6.2831853}%
%
\special{pn 13}%
\special{ar 2200 1200 400 200 4.3318826 4.7123890}%
%
\special{pn 13}%
\special{ar 2200 1200 400 200 3.1415927 3.9935590}%
%
\special{pn 13}%
\special{ar 2000 1200 200 100 6.2831853 3.1415927}%
%
\special{pn 13}%
\special{pa 2200 1000}%
\special{pa 2400 1000}%
\special{fp}%
\put(10.0000,-10.0000){\makebox(0,0){$\bullet$}}%
\put(24.0000,-10.0000){\makebox(0,0){$\bullet$}}%
\put(28.0000,-10.0000){\makebox(0,0){$\bullet$}}%
%
\special{pn 8}%
\special{pa 1700 700}%
\special{pa 1700 1300}%
\special{dt 0.045}%
\put(17.0000,-14.0000){\makebox(0,0){$t$}}%
%
\special{pn 13}%
\special{pa 2800 1000}%
\special{pa 3200 1000}%
\special{fp}%
%
\special{pn 13}%
\special{ar 3200 1000 200 200 1.5707963 2.8198421}%
%
\special{pn 13}%
\special{ar 3400 1000 200 200 6.2831853 1.5707963}%
%
\special{pn 13}%
\special{ar 3300 1000 300 200 3.3683915 6.2831853}%
%
\special{pn 13}%
\special{ar 3200 1100 200 100 5.6396842 1.5707963}%
%
\special{pn 13}%
\special{ar 3400 1100 200 100 2.4980915 4.7123890}%
%
\special{pn 13}%
\special{ar 3200 1100 200 100 4.7123890 4.9097845}%
%
\special{pn 13}%
\special{pa 3400 1000}%
\special{pa 3560 1000}%
\special{fp}%
%
\special{pn 13}%
\special{pa 3640 1000}%
\special{pa 3800 1000}%
\special{fp}%
\put(38.0000,-10.0000){\makebox(0,0){$\bullet$}}%
%
\special{pn 8}%
\special{pa 2900 700}%
\special{pa 2900 1300}%
\special{dt 0.045}%
\put(29.0000,-14.0000){\makebox(0,0){$t$}}%
%
\special{pn 13}%
\special{ar 3400 1100 200 100 1.5707963 1.7681919}%
\put(44.0000,-10.0000){\makebox(0,0){$\bullet$}}%
\put(52.0000,-10.0000){\makebox(0,0){$\bullet$}}%
\put(17.0000,-16.0000){\makebox(0,0){reducible}}%
\put(48.0000,-16.0000){\makebox(0,0){non-reducible}}%
%
\special{pn 13}%
\special{ar 4900 1000 500 200 3.1415927 6.2831853}%
%
\special{pn 13}%
\special{ar 4800 1000 600 200 6.2831853 3.1415927}%
%
\special{pn 13}%
\special{ar 4700 1000 500 200 5.2023463 6.2831853}%
%
\special{pn 13}%
\special{ar 4700 1000 500 200 3.1415927 4.7123890}%
\put(33.0000,-16.0000){\makebox(0,0){reducible}}%
\end{picture}}%
\caption{Reducible and non-reducible ropes}
\label{fig:reducible_ropes}
\end{center}
\end{figure}

The function $f(t):=\tan\pi(t-(1/2))$
gives an orientation preserving diffeomorphism $f\co(0,1)\xrightarrow{\cong}\R$.
Define the ``cut-off'' map $c\co R\to\psi$ by
\[
 c(r):=(f\times\id_{D^2})(r|_{(0,1)}).
\]
This map is defined since, for any reducible rope $r$, $c(r)$ has exactly one long component.

Our aim is to show that $c$ is a weak equivalence, and for this we introduce the following posets as interfaces between $R$ and $\psi$.

\begin{defn}\label{def:E}
Define a poset $\E$ by
\[
 \E:=\{(t,r)\in(0,1)\times R\mid
 r\text{ is reducible at }t\}.
\]
Define the partial order $\le$ on $\E$ so that $(t,r)<(t',r')$ if and only if $r=r'$ and $t<t'$.
We regard $\E$ as a small category in the same way as $\D$.
The total space of all morphisms is topologized as a subspace of $\bigl(\Delta\sqcup((0,1)\times(0,1)\setminus\Delta)\bigr)\times R$, where $\Delta$ is the diagonal set.

Define $\E^{\perp}$ as a subposet of $\E$ consisting of $(t,r)$ with $r$ strongly reducible at $t$.
\end{defn}

\begin{lem}
The simplicial spaces $N_*\E$ and $N_*\E^{\perp}$ are good.
\end{lem}
\begin{proof}
The same as the proof of Lemma~\ref{lem:good}.
\end{proof}

Any element in $N_l\E$ can be expressed as a pair $(t_0\le\dotsb\le t_l;r)$ where $0<t_i<1$ and $r\in R$ is reducible at each $t_i$.

\begin{prop}\label{prop:zig_zag2}
There exists a zig-zag of levelwise homotopy equivalences
$N_*\E\leftarrow N_*\E^{\perp}\to N_*\D^{\perp}$.
Consequently $B\E$ is weakly homotopy equivalent to $B\D$.
\end{prop}
\begin{proof}
That the inclusion $\E^{\perp}\to\E$ induces a homotopy equivalence $B\E^{\perp}\xrightarrow{\simeq}B\E$ follows in the same way as \cite[Theorem~3.9]{GalatiusRandal-Williams10}, using \cite[Lemma~3.4]{GalatiusRandal-Williams10}.

Define a functor $\Phi\co\E^{\perp}\to\D^{\perp}$ that induces a simplicial map $\Phi\co N_*\E^{\perp}\to N_*\D^{\perp}$ by
\[
 \Phi(t;r):=(f(t);c(r))
\]
(see Figure~\ref{fig:PhiGamma}).
\begin{figure}
\begin{center}
\input{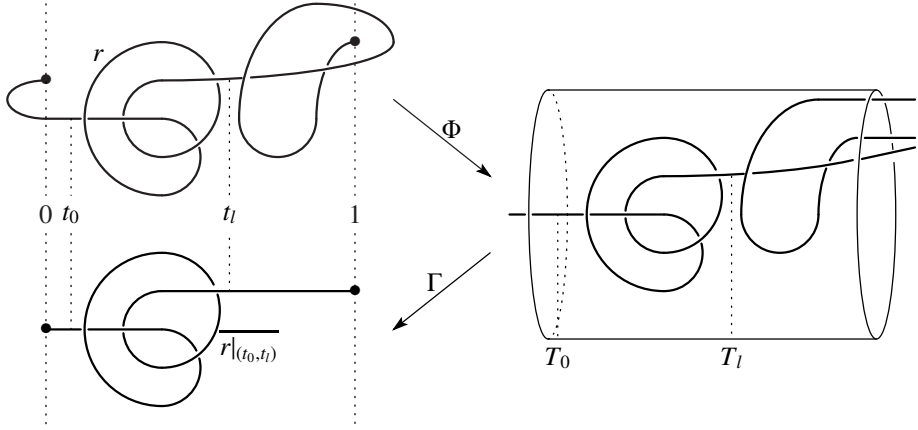}
\end{center}
\caption{The maps $\Phi$ and $\Gamma$}
\label{fig:PhiGamma}
\end{figure}
Define the map in the reverse direction $\Gamma\co N_l\D^{\perp}\to N_l\E^{\perp}$ by
\[
 \Gamma(T_0\le\dotsb\le T_l;M):=(t_0\le\dotsb\le t_l;(f^{-1}\times\id_{D^2})(\overline{M|_{[T_0,T_l]}})),
\]
where $\overline{M|_{[T_0,T_l]}}$ is the long-extension of $M|_{[T_0,T_l]}$ (see \eqref{eq:long-extension}),
and $t_i:=f^{-1}(T_i)\in(0,1)$ (see Figure~\ref{fig:PhiGamma}).
Notice that $(f^{-1}\times\id_{D^2})(M)$ is not necessarily a tame (or regular) submanifold of $(0,1)\times D^2$ for some $M\in\psi$
(for example, a manifold $M$ that is ``knotted'' outside arbitrary compact set of $\R^3$),
but $(f^{-1}\times\id_{\R^2})(\overline{M|_{[T_0,T_l]}})$ is indeed a tame submanifold in $(0,1)\times D^2$ since $\overline{M|_{[T_0,T_l]}}$ is a union of two straight half-lines outside $[T_0,T_l]\times D^2$.

We show that $\Phi$ is a levelwise homotopy equivalence, with a homotopy inverse $\Gamma$.
The composite $\Phi\circ\Gamma$ is given by
\[
 \Phi\circ\Gamma(T_0\le\dotsb\le T_l;M)=(T_0\le\dotsb\le T_l;\overline{M|_{[T_0,T_l]}})
\]
and a similar isotopy from the proof of Proposition~\ref{prop:zig-zag} proves that $\Phi\circ\Gamma\simeq\id$.

The other composite $\Gamma\circ\Phi$ is given by
\[
 \Gamma\circ\Phi(t_0\le\dotsb\le t_l;r):=(t_0\le\dotsb\le t_l;\overline{r|_{(t_0,t_l)}}),
\]
where
\[
 \overline{r|_{(t_0,t_l)}}:=
 \bigl([0,t_0]\times\{p_{23}(r|_{t_0})\}\bigr)\cup
 r|_{(t_0,t_l)}\cup
 \bigl([t_l,1]\times\{p_{23}(r|_{t_l})\}\bigr)\in R
\]
is the ``long-extension'' of $r|_{(t_0,t_l)}$.
The rope $\overline{r|_{(t_0,t_l)}}$ can be obtained from $r$ by ``unknotting'' the edge parts $r|_{(-\infty,t_0)}\sqcup r|_{(t_l,\infty)}$ (see Figure~\ref{fig:PhiGamma}).
This unknotting can be realized by applying Lemma~\ref{lem:W_L_contractible} and its analogue respectively  to $r_{[t_l,\infty)}$ and $r_{(-\infty,t_0]}$, keeping $r|_{[t_0,t_l]}$ unchanged (and hence keeping $r$ to be strongly reducible at each $t_i$).
Thus $\Gamma\circ\Phi\simeq\id$.
\end{proof}

\begin{lem}[Mostovoy {\cite[Lemma~10]{Mostovoy02}}]\label{lem:W_L_contractible}
Let $W$ be the subspace of $R$ consisting of $r$ that is ``strongly reducible'' at $0$, that means $r|_{(-\epsilon,\epsilon)}=r|_{[0,\epsilon)}=[0,\epsilon)\times\{p_{23}(\partial_0r)\}$ for some $\epsilon>0$.
Then $W$ is contractible.
In other words, there exists a canonical homotopy for any $r\in W$ that transforms $r$ to the trivial rope $[0,1]\times\{(0,0)\}$ keeping $r$ to be strongly reducible at $0$.
\end{lem}

\begin{proof}
Let $W'\subset W$ be the subspace consisting of $r\in W$ with $\partial_ir=(i,0,0)$ for $i=0,1$.
We show that the inclusion $W'\hookrightarrow W$ is a homotopy equivalence.
This completes the proof since $W'$ is homeomorphic to the space $W^0_L$ from \cite[Lemma~10]{Mostovoy02} via the diffeomorphism $\R^1\times D^2\xrightarrow{\approx}\R^3=\R^1\times\R^2$ defined by $(x,u)\mapsto(x,\tan(\pi\abs{u}/2)\cdot u)$, and $W^0_L$ has been shown to be contractible.
In the proof of \cite[Lemma~10]{Mostovoy02} the contracting homotopy (denoted by $D''_T$) keeps ropes to be strongly reducible at $0$.

A homotopy inverse $W\to W'$ can be realized as follows.
For $p\in\R^2$ let $\xi_p\co\R^2\to\R^2$ be the scaling by $1/2$ centered at $p$, namely $\xi_p(x):=(x+p)/2$.
Notice that if $p\in D^2$ then $\xi_p(D^2)\subset D^2$.
Let $b\co\R^1\to\R^1$ be a monotonically increasing $C^{\infty}$-function satisfying $b(x)=0$ for $x<1/3$ and $b(x)=1$ for $x>2/3$.
For $r\in W$, define $\Xi_r\co\R^1\times D^2\to\R^1\times D^2$ by
\begin{equation}\label{eq:Xi}
 \Xi_r(x,(y,z)):=(x,\xi_{-(1-b(x))p_{23}(\partial_0r)-b(x)p_{23}(\partial_1r)}(y,z)).
\end{equation}
Then $\Xi_r(r)\subset\R^1\times D^2$ and
$\Xi_r(\partial_ir)=(i,\xi_{-p_{23}(\partial_ir)}(p_{23}(\partial_ir)))=(i,0,0)$.
Moreover
$\Xi_r(r)$ is strongly reducible at $0$ because for a small $0<\epsilon<1/3$ such that $r|_{(-\epsilon,\epsilon)}=[0,\epsilon)\times\{p_{23}(\partial_0r)\}$ we have
\[
 \Xi_r(r)|_{(-\epsilon,\epsilon)}
 =\Xi_r([0,\epsilon)\times\{p_{23}(\partial_0r)\})
 =[0,\epsilon)\times\{\xi_{-p_{23}(\partial_0r)}(p_{23}(\partial_0r))\}
 =[0,\epsilon)\times\{(0,0)\}.
\]
Thus we have a continuous map $\Xi_{\bullet}\co W\to W'$.
The composite $W'\hookrightarrow W\xrightarrow{\Xi_{\bullet}}W'$ is the scaling by $1/2$ in the $(y,z)$-direction and is homotopic to $\id_{W'}$.
The other composite $W\xrightarrow{\Xi_{\bullet}}W'\hookrightarrow W$ is also homotopic to $\id_{W}$ because $\xi_p$ is homotopic to $\id_{D^2}$ for any $p\in D^2$.
\end{proof}

\begin{thm}\label{thm:BE=R}
The forgetful map induces a weak equivalence $v\co B\E\to R$.
\end{thm}
\begin{proof}
Replace $\D$ with $\E$ and take $a$ from $(0,1)$ in the proof of Theorem~\ref{thm:BK=psi_s}.
\end{proof}

\begin{cor}\label{cor:BK=R}
There exists a commutative diagram consisting of (weak) equivalences
\[
 \xymatrix{
  R \ar[r]^-{c}_-{\sim} & \psi & \\
  B\E^{\perp}\ar[u]^-{v'}_-{\sim}\ar[r]^-{\Phi}_-{\simeq} & B\D^{\perp}\ar[u]^-{u'}_-{\sim}\ar[r]^-{F}_-{\sim} & B\K
 }
\]
where $u',v'$ are the composites of $u,v$ with the inclusions.
\end{cor}

\subsection{Reducible ropes and Mostovoy's parametrized short ropes}
In Corollary~\ref{cor:BK=R} we have seen that $B\K$ is weakly equivalent to $R$.
The following Theorem solves affirmatively the conjecture of Mostovoy.
For a rope $r$ let $l(r)$ denote the length of $r$.

\begin{thm}\label{thm:ropes_equiv}
Let $B_2$ be the space of embeddings $\rho\co[0,1]\hookrightarrow\R^3$ satisfying $\rho(i)=(i,0,0)$ for $i=0,1$ and $l(\rho([0,1])<3$ (Mostovoy's (parametrized) \emph{short ropes} \cite{Mostovoy02}).
Then $B_2$ is weakly equivalent to $R$.
\end{thm}

The rest of this paper is devoted to the proof of Theorem~\ref{thm:ropes_equiv}.

It is not difficult to see that the image of any $\rho\in B_2$ is in $\R^1\times D^2(2\sqrt{2})$,
where $D^2(\tau)$ is the open $2$-disk centered at the origin and of radius $\tau$.
Thus we may write $B_2$ as
\[
 B_2=\{\rho\co[0,1]\hookrightarrow\R^1\times D^2(2\sqrt{2})\mid\rho(i)=(i,0,0)\text{ for }i=0,1\text{ and }l(\rho([0,1]))<3\}.
\]

Let $B_2^{\rm u}:=B_2\mathord{/}\mathrm{Diff}^+([0,1])$ (``u'' indicates ``unparametrized''), namely $B_2^{\rm u}$ is the space of ropes in $\R^1\times D^2(2\sqrt{2})$ with $\partial r=\{\partial_0r,\partial_1r\}$, $\partial_ir=(i,0,0)$ and $l(r)<3$.
The following holds since $\mathrm{Diff}^+([0,1])$ is contractible.

\begin{lem}\label{lem:B->B^2}
$B_2\to B_2^{\rm u}$ is a homotopy equivalence.
\end{lem}

We notice that $l(r)<3$ implies that $r$ is a reducible rope, and hence we may regard $B_2^{\rm u}$ as a subspace of $R(2\sqrt{2})$, where $R(\tau)$ is the space of reducible ropes in $\R^1\times D^2(\tau)$.

Let $R^{\rm s}(\tau)\subset R(\tau)$ be the subspace consisting of $r\in R(\tau)$ with $l(r)<3$ (``s'' indicates ``short'').
By definition $B_2^{\rm u}\subset R^{\rm s}(2\sqrt{2})$.

\begin{lem}\label{lem:B^u->R^2}
The inclusion $B_2^{\rm u}\hookrightarrow R^{\rm s}(2\sqrt{2})$ is a homotopy equivalence.
\end{lem}
\begin{proof}
For $r\in R^{\rm s}(2\sqrt{2})$, let $\Xi_r\co\R^1\times D^2(2\sqrt{2})\to\R^1\times D^2(2\sqrt{2})$ be the map defined in \eqref{eq:Xi} (notice that if $p\in D^2(\tau)$ then $\xi_p(D^2(\tau))\subset D^2(\tau)$).
Then $l(\Xi_r(r))<3$ because $\Xi_r$ is a shrinking map in the $(y,z)$-direction and hence does not increase the length, and $\Xi_r(\partial_ir)=(i,\xi_{-\partial_ir}(\partial_ir))=(i,0,0)$.
Thus we have a continuous map $\Xi_{\bullet}\co R^{\rm s}(2\sqrt{2})\to B_2^{\rm u}$.
The composite $B_2^{\rm u}\hookrightarrow R^{\rm s}(2\sqrt{2})\xrightarrow{\Xi_{\bullet}}B_2^{\rm u}$ is the scaling by $1/2$ in the $(y,z)$-direction and is homotopic to $\id_{B_2^{\rm u}}$.
The other composite $R^{\rm s}(2\sqrt{2})\xrightarrow{\Xi_{\bullet}}B_2^{\rm u}\hookrightarrow R^{\rm u}(2\sqrt{2})$ is also homotopic to $\id_{R^{\rm s}(2\sqrt{2})}$ because $\xi_p$ is homotopic to $\id_{D^2(\tau)}$ for any $p\in D^2(\tau)$.
\end{proof}

Next let $\E(\tau)$ be the poset consisting of $(t,r)$, where $t\in(0,1)$ and $r\in R(\tau)$ such that $r$ is reducible at $t$.
The partial order is defined in the same way as in Definition~\ref{def:E}.
Define $\E^{\rm s}(\tau)$ be a subposet of $\E(\tau)$ consisting of $(t,r)$ with $l(r)<3$.
Then we have a commutative diagram
\begin{equation}\label{eq:diagram}
\begin{split}
 \xymatrix{
  B\E^{\rm s}(2\sqrt{2})\ar[rr]^-v_-{\sim}\ar[d] & & R^{\rm s}(2\sqrt{2})\ar@{^(->}[d] & & B_2\ar[ll]_-{\simeq}^-{\text{Lemmas~\ref{lem:B->B^2}, \ref{lem:B^u->R^2}}} \\
  B\E(2\sqrt{2})\ar[rr]^-{\simeq}_-{\text{Theorem~\ref{thm:BE=R}}} & & R(2\sqrt{2})\ar[rr]^-{\approx} & & R
 }
\end{split}
\end{equation}
where $B\E^{\rm s}(2\sqrt{2})\to B\E(2\sqrt{2})$ and $v$ are induced respectively by the inclusion and the forgetful map (see Theorem~\ref{thm:BK=psi_s}).
That $v$ is a weak equivalence follows from the same argument as in the proof of Theorem~\ref{thm:BE=R}.
The homeomorphism $R=R(1)\xrightarrow{\approx}R(\tau)$ is given by $r\mapsto(\id_{\R^1}\times\overline{\tau})(r)$,
where $\overline{\tau}:D^2\xrightarrow{\approx}D^2(\tau)$ is the scalar multiplication by $\tau$.
The diagram \eqref{eq:diagram} together with the following Lemma completes the proof of Theorem~\ref{thm:ropes_equiv}.

\begin{lem}
$B\E^{\rm s}(\tau)\to B\E(\tau)$ is a homotopy equivalence.
\end{lem}
\begin{proof}
Let $\E^{\perp}(\tau)$ be the subposet of $\E(\tau)$ consisting of $(t,r)$ with $r$ strongly reducible at $t$, and $\E^{\perp{\rm s}}(\tau):=\E^{\perp}(\tau)\cap\E^{\rm s}(\tau)$.
Then the inclusion $\E^{\perp{\rm s}}(\tau)\hookrightarrow\E^{\rm s}(\tau)$ induces a homotopy equivalence $B\E^{\perp{\rm s}}(\tau)\xrightarrow{\simeq}B\E^{\rm s}(\tau)$.
This follows in the same way as Galatius and Randal-Williams \cite[Theorem~3.9]{GalatiusRandal-Williams10}, using \cite[Lemma~3.4]{GalatiusRandal-Williams10};
modifying $r$ to be strongly reducible at each $t$ can be done keeping the length less than $3$.

We show that $\E^{\perp{\rm s}}(\tau)\hookrightarrow\E^{\perp}(\tau)$ induces a levelwise homotopy equivalence $N_*\E^{\perp{\rm s}}(\tau)\to N_*\E^{\perp}(\tau)$.
A homotopy inverse $N_l\E^{\perp}(\tau)\to N_l\E^{\perp{\rm s}}(\tau)$ is given as follows;
firstly unknot $r|_{(-\infty,t_0]}\sqcup r|_{[t_l,\infty)}$ similarly to the proof of Lemma~\ref{lem:W_L_contractible} to obtain $\overline{r|_{(t_0,t_l)}}$, then shrink $\overline{r|_{(t_0,t_l)}}$ to
\[
 \Theta(t,r):=
 \theta_{t,r}(\overline{r|_{(t_0,t_l)}})\cup
 \bigl([l(\overline{r|_{(t_0,t_l)}})^{-1},1]\times\{p_{23}(r|_{t_l})/l(\overline{r|_{(t_0,t_l)}})\}\bigr),
\]
where $\theta_{t,r}\co\R^3\to\R^3$ is given by $\theta_{t,r}(x):=x/l(\overline{r|_{(t_0,t_l)}})$.
It can be seen that $l(\Theta(t,r))<3$ since $l(\theta_{t,r}(\overline{r|_{(t_0,t_l)}}))=1$.
The map $N_l\E^{\perp}(\tau)\to N_l\E^{\perp{\rm s}}(\tau)$,
\[
 (t_0\le\dots\le t_l;r)\mapsto(t_0/l(\overline{r|_{(t_0,t_l)}}),\dots,t_l/l(\overline{r|_{(t_0,t_l)}});\Theta(t,r))
\]
gives a levelwise homotopy inverse.
\end{proof}

\section*{Acknowledgments}
The authors are truly grateful to Tadayuki Watanabe for invaluable comments and discussions, and to Katsuhiko Kuribayashi for his support in starting this work.
The first author deeply thanks Thomas Goodwillie, Robin Koytcheff, Rustam Sadykov, and Victor Turchin for interesting discussions and comments.
The authors appreciate the referees for careful reading of the previous version of this manuscript and for many fruitful comments.
The first author is partially supported by JSPS KAKENHI Grant Number 26800037.
The second author is partially supported by JSPS KAKENHI Grant Numbers JP25800038 and JP16K05144.


\end{document}